\documentclass{amsart}
\usepackage{amsmath}
\usepackage{amssymb}
\usepackage{amstext}
\usepackage{bm}

\usepackage{enumerate}
\newtheorem{theorem}{Theorem}[section]
\newtheorem{lemma}[theorem]{Lemma}

\theoremstyle{definition}
\newtheorem{definition}[theorem]{Definition}

\theoremstyle{remark}
\newtheorem{remark}[theorem]{Remark}

\numberwithin{equation}{section}

\begin{document}

\title{Euler evolution of a concentrated vortex in planar bounded domains}

%    Remove any unused author tags.

%    author one information
\author{Daomin Cao, Guodong Wang}
\address{}
\curraddr{}
\email{}
\thanks{}

\subjclass[]{}


\date{}

\dedicatory{}

\begin{abstract}
In this paper, we consider the time evolution of an ideal fluid in a planar bounded domain. We prove that if the initial vorticity is supported in a sufficiently small region with diameter $\varepsilon$, then the time evolved vorticity is also supported in a small region with diameter $d$, $d\leq C\varepsilon^{\alpha}$ for any $\alpha<\frac{1}{3}$, and the center of the vorticity tends to the point vortex, the motion of which is described by the Kirchhoff-Routh equation.
\end{abstract}

\maketitle

\section{Introduction}
The motion of an incompressible inviscid fluid in the plane is governed by the well-known two-dimensional Euler fluid dynamical equation. When the vorticity consists of a sum of $k$ point vortices, the motion is described by the point vortex model, a Hamiltonian dynamical system called the Kirchhoff-Routh equation with the Kirchhoff-Routh function as the Hamiltonian, see \cite{L} for a general discussion. A natural problem is the connection between the Euler equation and the vortex model. More specifically, we ask the following questions: suppose that at time zero the initial vorticity is sufficiently concentrated in $k$ small regions in some sense,

(a) does the evolved vorticity of the Euler equation remain concentrated?

(b) if so, does the concentration point satisfy the Kirchhoff-Routh equation?

The problem is called desingularization of point vortices. In this paper, we answer the above two questions in the case of bounded domains. We prove that, for initial vorticity supported in a sufficiently small region, then the time evolved vorticity is also supported in a sufficiently small region, and the center of the vorticity converges to the solution of the Kirchhoff-Routh equation uniformly in any finite time interval.

Several significant results have been obtained in this respect. The case for short time and $k$ vortices without sign condition was proved in \cite{MP1}; the case for any time and a single vortex in bounded domains was proved in \cite{T}; the case for any time and two vortices with opposite signs in bounded domains was proved in \cite{MPa}; the case for any time and $k$ vortices with the same signs was proved in \cite{MP2}; the case for any time and $k$ vortices without sign condition was proved in \cite{MP3}.

Our result and method are closely related to \cite{T} and \cite{MP3}. In \cite{T}, Turkington partially solved the desingularization problem. He proved that starting with a blob vorticity concentrating at some given point, then most part of the evolved vorticity remains concentrated, and the center of the vorticity converges to the solution of the Kirchhoff-Routh equation. But the concentration there is in the sense of distribution, and whether the support of the evolved vorticity shrinks to a point is unknown. In fact, the method used in \cite{T} is based on energy argument, which is too rough to be used to analyze the location of the support. In \cite{MP3}, Marchioro and Pulvirenti  developed a new method to deal with this problem. Roughly speaking, they showed that, for a single blob of vorticity moving in an external force in all of $\mathbb{R}^2$, initially supported in a small region of diameter $\varepsilon$, the radial velocity that takes the fluid particles away from the center of the vorticity vanishes as $\varepsilon\rightarrow 0$, so the fluid particles on the support of the vorticity must be contained in a small disk with radius $d(\varepsilon)$, $d(\varepsilon)\rightarrow0$ as $\varepsilon\rightarrow0$, in any finite time interval. Based the same idea, Marchioro in \cite{M} improved the result in \cite{MP3} by giving a better estimate on the size of the support of the vorticity under a weaker assumption on the initial data.

In this paper, we improve the result in \cite{T}, showing that if the initial vorticity is supported in a sufficiently small region with diameter $\varepsilon$, then the time evolved vorticity is also supported in a small region, the diameter of which vanishes as $\varepsilon\rightarrow0$, and the center of the vorticity tends to a point, the motion of which is described by the Kirchhoff-Routh equation.
 The basic idea to prove our result is mostly inspired by \cite{MP3}, where the motion of $k$ concentrated vortices in all $\mathbb{R}^2$ is considered.  The idea is as follows. Firstly we introduce a regularized system, in which case the external force is smooth and bounded even if the support of the vorticity approaches the boundary. Then we repeat the argument in \cite{MP3}(or\cite{M}) for the regularized system, and find that the vorticity in the regularized system is supported in a sufficiently small disk, the center of which satisfies the Kirchhoff-Routh equation. But a single point vortex will never touch the boundary of the domain, so the support of the vorticity for the regularized system is away from the boundary if $\varepsilon$ is sufficiently small, in which case the vorticity of the regularized system coincides with the one of the Euler equation, which concludes the proof.

Our method of dealing with the boundary effect can also be used to study the evolution of $k$ vortices, which will be discussed briefly in Section 4..

\section{Main result}
Firstly we introduce some notations for later use.
Let $D\subset \mathbb R^2$ be a bounded domain with smooth boundary, and $G$ be the Green function for $-\Delta$ in $D$ with zero
Dirichlet boundary condition, which can be written as

\begin{equation}
G(x,y)=\Gamma(x,y)-h(x,y), \,\,\,x,y\in
D,
\end{equation}
where $\Gamma(x,y)=-\frac{1}{2\pi}\ln |x-y|$, $h(x,y)$ is the regular part of $G$. Also, we define the Kirchhoff-Routh function for $N=1$ to be

\begin{equation}
H(x)=\frac{1}{2}h(x,x), \,\,\,x\in
D.
\end{equation}

Throughout this paper, $B_r(x)$ denotes the open disk of center $x$ and radius $r$, $dist(\cdot,\cdot)$ denotes the distance of two points or sets, $D_\rho$ is an open subset of $D$ defined by $D_\rho\triangleq \{x\in D|dist(x,\partial D)>\rho\}$, $|A|$ denotes the Lebesgue measure for some measurable set $A\subset \mathbb{R}^2$, $supp f$ denotes the support of some function $f$, and $J(v_1,v_2)=(v_2,-v_1)$ denotes clockwise rotation through $\frac{\pi}{2}$ for any vector $(v_1,v_2)\in \mathbb{R}^2$.

Consider the motion of an ideal fluid in $D$, which is governed by the following Euler equation:

\begin{equation}\label{1}
\begin{cases}
  \partial_t\mathbf{v}(x,t)+(\mathbf{v}\cdot\nabla)\mathbf{v}(x,t)=-\nabla P(x,t) \,\,\,\,\,\,\,\,\,\,\text{in $D\times(0,+\infty)$},\\
  \nabla\cdot\mathbf{v}(x,t)=0 \,\,\,\,\,\,\,\,\,\,\,\,\,\,\,\,\,\,\,\,\,\,\,\,\,\,\,\,\,\,\,\,\,\,\,\,\,\,\,\,\,\,\,\,\,\,\,\,\,\,\,\,\,\,\,\,\,\,\,\,\,\,\,\,\,\,\,\,\,\,\,\,\text{in $D\times(0,+\infty)$},\\
 \mathbf{v}(x,0)=\mathbf{v}_0(x)\,\,\,\,\,\,\,\,\,\,\,\,\,\,\,\,\,\,\,\,\,\,\,\,\,\,\,\,\,\,\,\,\,\,\,\,\,\,\,\,\,\,\,\,\,\,\,\,\,\,\,\,\,\,\,\,\,\,\,\,\,\,\,\,\,\,\,\,\,\text{in $D$},
 \\ \mathbf{v}(x,t)\cdot \vec{n}(x)=0 \,\,\,\,\,\,\,\,\,\,\,\,\,\,\,\,\,\,\,\,\,\,\,\,\,\,\,\,\,\,\,\,\,\,\,\,\,\,\,\,\,\,\,\,\,\,\,\,\,\,\,\,\,\,\,\,\,\,\,\,\,\,\,\,\,\text{on $\partial D\times(0,+\infty)$},
\end{cases}
\end{equation}
where $\mathbf{v}=(v_1,v_2)$ is the velocity field, $P$ is the pressure, $\mathbf{v}_0$ is the initial velocity field and $\vec{n}$ is the outward unit normal of $\partial D$. Here we assume that the fluid density is one, and impose the impermeability boundary condition.

To give the vorticity form of \eqref{1}, we calculate formally in the following. By introducing the vorticity $\omega = \partial_1 v_2-\partial_2v_1$ and using the identity $\frac{1}{2}\nabla|\mathbf{v}|^2=(\mathbf{v}\cdot\nabla)\mathbf{v}+J\mathbf{v}\omega$,
the first equation of $\eqref{1}$ becomes

\begin{equation}\label{2}
 \partial_t\mathbf{v}+\nabla(\frac{1}{2}|\mathbf{v}|^2+P)-J\mathbf{v}\omega=0.
\end{equation}
Taking the curl in $\eqref{2}$ we have

\begin{equation}\label{3}
\partial_t\omega(x,t)+(\mathbf{v}\cdot\nabla)\omega(x,t)=0 \,\,\,\,\,\,\,\,\,\,\,\,\,\,\,\,\text{in $D\times(0,+\infty)$},
\end{equation}
which means that the vorticity moves along the trajectory of the fluid particle. \eqref{3} is a transport equation, the solution of which can be written as

 \begin{equation}\label{21}
   \omega(x,t)=\omega(\Phi_{-t}(x),0)
 \end{equation}
where $\Phi_{t}(x)$ is the position of the fluid particle at time $t$ with initial position $x$, defined by

\begin{equation}\label{22}
\begin{cases}
 \frac{d}{dt}\Phi_t(x)=\mathbf{v}(\Phi_t(x),t),
 \\ \Phi_0(x)=x.
\end{cases}
\end{equation}
Since $\mathbf{v}$ is divergence-free and $\mathbf{v}\cdot \vec{n}=0$ on $\partial D$, $\mathbf{v}$ can be expressed in terms of $\omega$

\begin{equation}\label{23}
\mathbf{v}=J\nabla G\omega
\end{equation}
 where $G\omega(x,t)=\int_DG(x,y)\omega(y,t)dy$. From \eqref{21},\eqref{22} and \eqref{23} we can define the weak solution to the Euler equation as follows.

 \begin{definition} Let $\omega_0\in L^\infty(D)$. We call the map $t\rightarrow (\omega(\cdot,t),\mathbf{v}(\cdot,t),\Phi_t(\cdot))$ a weak solution of the Euler solution with initial vorticity $\omega_0$ if we have

 \begin{equation}\label{101}
 \omega\in L^\infty((0,+\infty),L^\infty(D)),\,\,\,\,\,\,\,\,\,  \mathbf{v}=G\omega\in C([0,+\infty)\times D),
 \end{equation}

 \begin{equation}\label{102}
 \omega(x,t)=\omega_0(\Phi_{-t}(x)),
\end{equation}

 \begin{equation}\label{103}
 \frac{d}{dt}{\Phi}_t(x)={\mathbf{v}}({\Phi}_t(x),t),\,\,\,\,\,\ {\Phi}_0(x)=x,\,\,\,\forall x\in D.
 \end{equation}
\end{definition}

It is well known that for any $\omega_0\in L^\infty(D)$, there is a unique weak solution $(\omega(x,t),\mathbf{v}(x,t),\Phi_t(x))$ satisfying \eqref{101}, \eqref{102} and \eqref{103}. Moreover, for all $t\geq0$, $\Phi_t$ is a homeomorphism from $D$ to $D$ which preserves Lebesgue measure and the distributional function of $\omega(\cdot,t)$ does not change with time, that is, for any $a\in \mathbb{R}$ and $t\geq 0$, we have
\begin{equation}\label{25}
|\{x\in D|\omega(x,t)>a\}|=|\{x\in D|\omega(x,0)>a\}|.
\end{equation}
Many proofs of this existence and uniqueness result can be found in the literature, see\cite{MP4} or \cite{Y} for example.

Now we consider a family of initial data $\omega^\varepsilon(x,0)\in L^\infty(D)$ satisfying
\begin{equation}\label{26}
\int_D\omega^\varepsilon(x,0)dx=1,
\end{equation}
\begin{equation}\label{27}
0\leq\omega^\varepsilon(x,0)\leq M\varepsilon^{-\eta},
\end{equation}
\begin{equation}\label{28}
supp\omega^\varepsilon(x,0)\subset B_\varepsilon(z_0),
\end{equation}
where $M$ and $\eta$ are fixed positive numbers and $z_0\in D$ is given.

For fixed $\varepsilon$, $\omega^\varepsilon(x,0)\in L^\infty(D)$, thus there exists a unique weak solution $\omega^\varepsilon(x,t)$, the time evolution of $\omega^\varepsilon(x,0)$ according to the Euler equation. Since the support of initial vorticity shrinks to a given point $z_0$ as $\varepsilon\rightarrow 0$, we ask whether the support of $\omega^\varepsilon(x,t)$ shrinks to a point; if it does, which point it will shrink to.

The main purpose of this paper is to answer this question. To be more precise, we will prove the following result:

\begin{theorem}\label{33}
Let $\omega^\varepsilon(x,t)$ be the unique weak solution of the Euler equation with initial vorticity $\omega^\varepsilon(x,0)$  satisfying \eqref{26},\eqref{27},\eqref{28}, and $T>0$ be fixed. Then for any $0<\alpha<\frac{1}{3}$, there exists $C>0$ depending only on $\alpha$ and $T$ such that
\begin{equation}\label{29}
supp\omega^\varepsilon(x,t)\subset B_{C\varepsilon^\alpha}(z(t)),\,\,\forall t\in[0,T],
\end{equation}
 where $z(t)$ is the solution of the following Kirchhoff-Routh equation
\begin{equation}\label{30}
\begin{cases}
 \frac{dz(t)}{dt}=-J\nabla H(z(t)),
 \\ z(0)=z_0.
\end{cases}
\end{equation}
\end{theorem}

\begin{remark}
From Theorem \ref{33} it is easy to see that $\omega^\varepsilon(x,t)$ converges to the Dirac measure with unit mass at $z(t)$ uniformly as $\varepsilon\rightarrow 0$ in the sense of distribution, that is, for any $\phi\in C_c^\infty(D)$, we have
\begin{equation}
\lim_{\varepsilon\rightarrow0}\int_D\phi(x)\omega^\varepsilon(x,t)dx=\phi(z(t)).
\end{equation}
In fact, by \eqref{25} we have $\int_D|\omega^\varepsilon(x,t)|dx=\int_D|\omega^\varepsilon(x,0)|dx=1$ for all $t\geq0$, then
\begin{equation}
\begin{split}
|\int_D\phi(x)\omega^\varepsilon(x,t)dx-\phi(z(t))|&=|\int_D(\phi(x)-\phi(z(t)))\omega^\varepsilon(x,t)dx|\\
&\leq\int_D|(\phi(x)-\phi(z(t)))\omega^\varepsilon(x,t)|dx\\
&=\int_{B_{C\varepsilon^\alpha}(z(t))}|(\phi(x)-\phi(z(t)))\omega^\varepsilon(x,t)|dx\\
&\leq \sup_{x\in B_{C\varepsilon^\alpha}(z(t))}|\phi(x)-\phi(z(t))|\int_D|\omega^\varepsilon(x,t)|dx\\
&=\sup_{x\in B_{C\varepsilon^\alpha}(z(t))}|\phi(x)-\phi(z(t))|.
\end{split}
\end{equation}
Since $C\varepsilon^\alpha\rightarrow 0$ uniformly as $\varepsilon\rightarrow 0$, the result follows from the uniform continuity of $\phi$.
\end{remark}

\begin{remark}
By the theory of ordinary differential equations, there exists a unique smooth solution to \eqref{30}. At the same time, for such a solution
 \begin{equation}
 \frac{d}{dt}H(z(t))=\nabla H(z(t))\cdot \frac{dz(t)}{dt}=-\nabla H(z(t))\cdot J\nabla H(z(t)=0,
 \end{equation}
so we have $H(z(t))=H(z(0))$ for all $t\geq 0$. On the other hand, $H(x)\rightarrow+\infty $ as $dist(x,\partial D)\rightarrow 0$, so $z(t)$ will be away from the boundary for all $t\geq 0$. Therefore, we can choose a number $\rho_0>0$, depending only on $z_0$ and $D$, such that $dist(z(t),\partial D)>\rho_0$ for all $t\geq 0$.
\end{remark}

\section{Proof of Main Result}
In this section we prove Theorem \ref{33}. Roughly speaking, the proof consists of three steps: 1, we introduce the regularized system and study its property; 2, we prove a localization lemma for the regularized system; 3, we show that the support of the vorticity in the regularized system does not approach the boundary, therefore the vorticity of the regularized system coincides with the one of the Euler equation.

\subsection{The Regularized System}

In bounded domain $D$, the boundary effect on the fluid is the term $J\nabla\int_Dh(x,y)\omega^\varepsilon(x,t)dy$, which is called the boundary force in this paper. If the support of $\omega^\varepsilon$ approaches the boundary, the boundary force will become singular, which makes the problem difficult. To overcome this difficulty, we introduce a regularized system from the original one.

We first define two functions $\theta$ and $\chi$ for later use, where $\theta$ satisfies:
\begin{equation}\label{61}
\begin{cases}
 \theta\in C_c^\infty(D), 0\leq\theta\leq1,
\\  supp\theta\subset D_{\frac{\rho_0}{3}},
\\  \theta(x)=1, \forall \,x\in D_{\frac{\rho_0}{2}},
\end{cases}
\end{equation}
and $\chi(x)$ satisfies:
\begin{equation}\label{62}
\begin{cases}
 \chi\in C_c^\infty(D),0\leq\chi\leq1;
\\  \chi(x)=1, \forall \,x\in D_{\frac{\rho_0}{10}}.
\end{cases}
\end{equation}
Existence of such functions can easily be obtained by using standard mollifying technique.

Now consider the following system in the whole plane $\mathbb{R}^2$:

 \begin{equation}\label{35}
   \bar{\omega}(x,t)=\bar{\omega}(\bar{\Phi}_{-t}(x),0),
 \end{equation}
\begin{equation}\label{36}
\frac{d}{dt}\bar{\Phi}_t(x)=\bar{\mathbf{v}}(\bar{\Phi}_t(x),t),\,\,\,\,\,\ \bar{\Phi}_0(x)=x,\,\,\, x\in \mathbb{R}^2,
\end{equation}
\begin{equation}\label{37}
\bar{\mathbf{v}}=J\nabla \bar{G}\bar{\omega},
\end{equation}
where $\bar{G}\bar{\omega}$ is defined by
\begin{equation}\label{38}
  \bar{G}\bar{\omega}(x,t)=\int_{\mathbb{R}^2}\Gamma(x,y)\bar{\omega}(y,t)dy-\int_{\mathbb{R}^2}h(x,y)\theta(x)\chi(y)\bar{\omega}(y,t)dy.
\end{equation}
Here we extend $h(x,y)$ to $\mathbb{R}^4$ by setting $h(x,y)=0$ if $x\notin D$ or $y\notin D$. In this case $h(x,y)\theta(x)\chi(y)\in C_c^\infty(\mathbb{R}^4)$.

\begin{remark}
Since the regularized system is defined on the whole plane, it is possible that a fluid particle, even with initial position within $D$, may move out of $D$. But we will prove that for any fluid particle with initial position on the set $supp\omega^\varepsilon(x,0)$, its trajectory will coincide with the one governed by the Euler equation with the same initial position.
\end{remark}
The definition of the weak solution to the regularized is exactly the same as the Euler equation. More specifically, we have the following lemma.

\begin{lemma}
For fixed $\varepsilon$, there exists a unique weak solution to the regularized system with initial vorticity $\omega^\varepsilon(x,0)$ satisfying \eqref{26},\eqref{27},\eqref{28}, that is, a map $t\rightarrow (\bar{\omega}^\varepsilon(\cdot,t),\bar{\mathbf{v}}^\varepsilon(\cdot,t),\bar{\Phi}_t^\varepsilon(\cdot))$ satisfying
 \begin{equation}\label{101}
 \bar{\omega}^\varepsilon\in L^\infty((0,+\infty),L^1(\mathbb{R}^2)\cap L^\infty(\mathbb{R}^2)),\,\,\,\,\,\,\,\,\,  \bar{\mathbf{v}}^\varepsilon=\bar{G}\bar{\omega}^\varepsilon\in C([0,+\infty)\times \mathbb{R}^2),
 \end{equation}
 \begin{equation}\label{102}
 \bar{\omega}^\varepsilon(x,t)={\omega}^\varepsilon(\bar{\Phi}^\varepsilon_{-t}(x),0),
\end{equation}
\begin{equation}\label{103}
 \frac{d}{dt}\bar{\Phi}^\varepsilon_t(x)=\bar{\mathbf{v}}^\varepsilon(\bar{\Phi}^\varepsilon_t(x),t),\,\,\,\,\,\ \bar{\Phi}^\varepsilon_0(x)=x,\,\,\,\forall x\in D.
 \end{equation}
Moreover, for all $t\geq0$, $\bar{\Phi}^\varepsilon_t$ is a homeomorphism from $\mathbb{R}^2$ to $\mathbb{R}^2$ which preserves Lebesgue measure and the distributional function of $\bar{\omega}^\varepsilon(\cdot,t)$ does not change with time, that is, for any $a\in \mathbb{R}$ and $t\geq 0$,
\begin{equation}\label{110}
|\{x\in D|\bar{\omega}^\varepsilon(x,t)>a\}|=|\{x\in D|\bar{\omega}^\varepsilon(x,0)>a\}|.
\end{equation}

\end{lemma}
\begin{proof}
The proof is exactly the same as the one for the Euler equation(see \cite{MB}, Chapter 8 for example), therefore we omit it here.
\end{proof}

For simplicity we write $\mathbf{F}^\varepsilon(x,t)=J\nabla\int_Dh(x,y)\theta(x)\chi(y)\bar{\omega}^\varepsilon(y,t)dy$. From now on we regard $\mathbf{F}^\varepsilon(x,t)$ as an external force. It is easy to see that for fixed $\varepsilon$ and $t$, $\mathbf{F}^\varepsilon(x,t)$ is a smooth and divergence-free vector field in $\mathbb{R}^2$ with compact support. Besides, we need some uniform estimates for $\mathbf{F}^\varepsilon$ as $\varepsilon\rightarrow 0$.

\begin{lemma}\label{444}
$\mathbf{F}^\varepsilon(x,t)$ is uniformly bounded and satisfies the uniform Lipschitz condition. More specifically,
there exist $L_1,L_2>0$, independent of $\varepsilon$, such that for all $\varepsilon>0, x,y\in \mathbb{R}^2, t\geq0$, we have
\begin{equation}\label{40}
|\mathbf{F}^\varepsilon(x,t)|<L_1,
\end{equation}
\begin{equation}\label{41}
|\mathbf{F}^\varepsilon(x,t)-\mathbf{F}^\varepsilon(y,t)|<L_2|x-y|.
\end{equation}
\end{lemma}
\begin{proof}
We write $\bar{h}(x,y)=h(x,y)\theta(x)\chi(y)\in C_c^\infty(\mathbb{R}^4)$. For any multi-index  $\alpha=(\alpha_1,\alpha_2)$, where $\alpha_1, \alpha_2$ are two non-negative integers, we have
\begin{equation}\label{42}
D_x^\alpha \int_D\bar{h}(x,y)\bar{\omega}^\varepsilon(y,t)dy= \int_DD_x^\alpha\bar{h}(x,y)\bar{\omega}^\varepsilon(y,t)dy,
\end{equation}
where $D^\alpha =\frac{\partial^{\alpha_1+\alpha_2}}{\partial_{x_1}^{\alpha_1}\partial_{x_2}^{\alpha_2}}$.
So
\begin{equation}\label{44}
|D_x^\alpha \int_D\bar{h}(x,y)\bar{\omega}^\varepsilon(y,t)dy|= |\int_DD_x^\alpha\bar{h}(x,y)\bar{\omega}^\varepsilon(y,t)dy|\leq C\int_D|\bar{\omega}^\varepsilon(y,t)|dy=C.
\end{equation}
Here and in the sequel $C$ denotes various positive numbers not depending on $\varepsilon$. Then
\begin{equation}
\begin{split}
|\mathbf{F}^\varepsilon(x,t)|&=|J\nabla\int_D\bar{h}(x,y)\bar{\omega}^\varepsilon(y,t)dy|\\
&=|\nabla\int_D\bar{h}(x,y)\bar{\omega}^\varepsilon(y,t)dy|\\
&\leq C.
\end{split}
\end{equation}
Also,
\begin{equation}
\begin{split}
|\mathbf{F}^\varepsilon(x,t)-\mathbf{F}^\varepsilon(y,t)|&=|J\nabla\int_D\bar{h}(x,z)\bar{\omega}^\varepsilon(z,t)dz-J\nabla\int_D\bar{h}(y,z)\bar{\omega}^\varepsilon(z,t)dz|\\
&=|\nabla\int_D\bar{h}(x,z)\bar{\omega}^\varepsilon(z,t)dz-\nabla\int_D\bar{h}(y,z)\bar{\omega}^\varepsilon(z,t)dz|\\
&\leq \sup_{x\in \mathbb{R}^2}|D_x^2\int_D\bar{h}(x,y)\bar{\omega}^\varepsilon(y,t)dy||x-y|\\
&\leq C|x-y|,
\end{split}
\end{equation}
which completes the proof.

\end{proof}

\begin{lemma}
For fixed $\varepsilon$ and any $t\in[0,+\infty)$, $\bar{\omega}^\varepsilon(\cdot,t)$ has compact support. Moreover, for any $f\in C^\infty(\mathbb{R}^2)$ we have
\begin{equation}\label{120}
\frac{d}{dt}\int_{\mathbb{R}^2}\bar{\omega}^\varepsilon(x,t)f(x)dx=\int_{\mathbb{R}^2}\bar{\omega}^\varepsilon(x,t)(\bar{\mathbf{v}}^\varepsilon\cdot\nabla f)(x,t)dx,
\end{equation}
where $\bar{\mathbf{v}}^\varepsilon$ is defined in \eqref{101}.
\end{lemma}
\begin{proof}
We use the following well-known inequality(see \cite{MB}, Chapter 8 for example):
\begin{equation}
|\bar{\mathbf{v}}^\varepsilon(\cdot,t)|_{L^\infty(\mathbb{R}^2)}\leq C(|\omega^\varepsilon(\cdot,0)|_{L^\infty(\mathbb{R}^2)}+|\omega^\varepsilon(\cdot,0)|_{L^\infty(\mathbb{R}^2)}),
 \end{equation}
 which means that $\bar{\mathbf{v}}^\varepsilon(\cdot,t)$ is uniformly bounded for $t\in[0,+\infty)$. On the other hand, the initial vorticity has compact support and by \eqref{102} the vorticity is constant along the particle paths, we know that $\bar{\omega}^\varepsilon(x,t)$ has compact support in any finite interval.

 \eqref{120} can be obtained by direct calculation. Indeed,
 \begin{equation}\label{124}\begin{split}
 \frac{d}{dt}\int_{\mathbb{R}^2}\bar{\omega}^\varepsilon(x,t)f(x)dx=&\frac{d}{dt}\int_{\mathbb{R}^2}\omega^\varepsilon(\bar{\Phi}_{-t}^\varepsilon(x),0)f(x)dx\\
 =&\frac{d}{dt}\int_{\mathbb{R}^2}\omega^\varepsilon(x,0)f(\bar{\Phi}_{t}^\varepsilon(x))dx\\
 =&\int_{\mathbb{R}^2}\omega^\varepsilon(x,0)\nabla f(\bar{\Phi}_{t}^\varepsilon(x))\cdot \bar{\mathbf{v}}^\varepsilon(\bar{\Phi}_{t}^\varepsilon(x),t) dx\\
 =&\int_{\mathbb{R}^2}\omega^\varepsilon(\bar{\Phi}_{-t}^\varepsilon(x),0)\nabla f(x)\cdot \bar{\mathbf{v}}^\varepsilon(x,t) dx\\
 =&\int_{\mathbb{R}^2}\bar{\omega}^\varepsilon(x,t)\nabla f(x)\cdot \bar{\mathbf{v}}^\varepsilon(x,t) dx.
 \end{split}
 \end{equation}
This completes the proof.

\end{proof}

\subsection{Localization Lemma}
Now we prove the localization property of $\bar{\omega}^\varepsilon(x,t)$ by taking the regularized boundary force $\mathbf{F}^\varepsilon(x,t)$ as an external force.

\begin{lemma}\label{47}
Let $T>0$ be fixed. Then for any $0<\alpha<\frac{1}{3}$, there exists $C>0$ depending on $\alpha$ and $T$ such that
\begin{equation}\label{245}
supp\bar{\omega}^\varepsilon(x,t)\subset B_{C\varepsilon^\alpha}(\bar{z}^\varepsilon(t)),\,\,\forall t\in[0,T],
\end{equation}
where $\bar{z}^\varepsilon(t)$ satisfies the following equation
\begin{equation}\label{46}
\begin{cases}
 \frac{d\bar{z}^\varepsilon(t)}{dt}=-\mathbf{F}^\varepsilon(\bar{z}^\varepsilon(t),t),
 \\ \bar{z}^\varepsilon(0)=z_0.
\end{cases}
\end{equation}
\end{lemma}
\begin{proof}
The proof is essentially the same as Theorem 2.1 in \cite{M} but replacing a given external force by a family of external forces $\mathbf{F}^\varepsilon$, which is uniform bounded and Lipschitz continuous.

Firstly we show that the distance between $\bar{z}^\varepsilon(t)$ and the center of the vorticity vanishes at least as the order $\varepsilon$, where the center of vorticity is defined by
\begin{equation}\label{111}
m^\varepsilon(t)\triangleq \int_{\mathbb{R}^2}x\bar{\omega}^\varepsilon(x,t)dx.
\end{equation}
We also define the moment of inertia of the vorticity with respect to $m^\varepsilon(t)$ by
\begin{equation}\label{112}
I^\varepsilon(t)\triangleq \int_{\mathbb{R}^2}|x-m^\varepsilon(t)|^2\bar{\omega}^\varepsilon(x,t)dx.
\end{equation}
We claim that there exists a positive number $C$ not depending on $\varepsilon$ such that for any $t\in[0,T]$
\begin{equation}
|\bar{z}^\varepsilon(t)-m^\varepsilon(t)|\leq C\varepsilon.
\end{equation}
In fact, by choosing $f(x)=x_i$, $i=1,2$, and using \eqref{120},
\begin{equation}\label{130}
\begin{split}
\frac{d}{dt}m^\varepsilon(t)&=\frac{d}{dt}\int_{\mathbb{R}^2}x\bar{\omega}^\varepsilon(x,t)dx\\
&=\frac{d}{dt}\int_{\mathbb{R}^2}\bar{\omega}^\varepsilon(x,t)\bar{\mathbf{v}}^\varepsilon(x,t)dx\\
&=\int_{\mathbb{R}^4}-\frac{1}{2\pi}\frac{J(x-y)}{|x-y|^2}\bar{\omega}^\varepsilon(x,t)
\bar{\omega}^\varepsilon(y,t)dxdy-\int_{\mathbb{R}^2}\bar{\omega}^\varepsilon(x,t)\mathbf{F}^\varepsilon(x,t)dx\\
&=-\int_{\mathbb{R}^2}\bar{\omega}^\varepsilon(x,t)\mathbf{F}^\varepsilon(x,t)dx,
\end{split}
\end{equation}
where we have used the antisymmetry of $-\frac{1}{2\pi}\frac{J(x-y)}{|x-y|^2}$. Similarly,
\begin{equation}\label{131}
\frac{d}{dt}I^\varepsilon(t)=-2\int_{\mathbb{R}^2}(x-m^\varepsilon(t))\cdot\mathbf{F}^\varepsilon(x,t)\bar{\omega}^\varepsilon(x,t)dx.
\end{equation}
Now
\begin{equation}\label{132}
\begin{split}
|\frac{d}{dt}I^\varepsilon(t)|&=|2\int_{\mathbb{R}^2}(x-m^\varepsilon(t))\cdot\mathbf{F}^\varepsilon(x,t)\bar{\omega}^\varepsilon(x,t)dx|\\
&=|2\int_{\mathbb{R}^2}(x-m^\varepsilon(t))\cdot\mathbf{F}^\varepsilon(x,t)\bar{\omega}^\varepsilon(x,t)dx-
2\int_{\mathbb{R}^2}(x-m^\varepsilon(t))\cdot\mathbf{F}^\varepsilon(m^\varepsilon(t),t)\bar{\omega}^\varepsilon(x,t)dx|\\
&\leq 2L_2\int_{\mathbb{R}^2}|x-m^\varepsilon(t))|^2|\bar{\omega}^\varepsilon(x,t)|dx\\
&=2L_2I^\varepsilon(t),
\end{split}
\end{equation}
where we have used the uniform Lipschitz condition \eqref{41} and the fact that $\int_{\mathbb{R}^2}(x-m^\varepsilon(t))\cdot\mathbf{F}^\varepsilon(m^\varepsilon(t),t)\bar{\omega}^\varepsilon(x,t)dx=0$.
Then by Gr\"{o}nwall's inequality we have
\begin{equation}
I^\varepsilon(t)\leq I^\varepsilon(0)e^{2L_2t}.
\end{equation}
On the other hand,
\begin{equation}
I^\varepsilon(0)=\int_{\mathbb{R}^2}{\omega}^\varepsilon(x,0)|x-m^\varepsilon(0)|^2\leq 4\varepsilon^2,
\end{equation}
so
\begin{equation}\label{133}
I^\varepsilon(t)\leq 4\varepsilon^2e^{2L_2T},\,\,\forall t\in[0,T]
\end{equation}
Now we calculate
\begin{equation}\label{135}
\begin{split}
|\bar{z}^\varepsilon(t)-m^\varepsilon(t)|=&|z_0-\int_0^t\mathbf{F}^\varepsilon(\bar{z}^\varepsilon(s),s)ds-m^\varepsilon(0)+
\int_0^t\int_{\mathbb{R}^2}\mathbf{F}^\varepsilon(x,s)\bar{\omega}^\varepsilon(x,s)dxds|\\
\leq&|z_0-m^\varepsilon(0)|+\int_0^t|\mathbf{F}^\varepsilon(\bar{z}^\varepsilon(s),s)-\mathbf{F}^\varepsilon(m^\varepsilon(s),s)|ds\\
+&\int_0^t|\int_{\mathbb{R}^2}(\mathbf{F}^\varepsilon(m^\varepsilon(s),s)-\mathbf{F}^\varepsilon(x,s))\bar{\omega}^\varepsilon(x,s)dx|ds\\
\leq&|z_0-m^\varepsilon(0)|+\int_0^tL_2|\bar{z}^\varepsilon(s)-m^\varepsilon(s)|ds\\
+&\int_0^t\int_{\mathbb{R}^2}L_2|m^\varepsilon(s)-x|\bar{\omega}^\varepsilon(x,s)dxds\\
\leq&|z_0-m^\varepsilon(0)|+\int_0^tL_2|\bar{z}^\varepsilon(s)-m^\varepsilon(s)|ds
+L_2 T \sup_{t\in[0,T]}|I^\varepsilon(t)|^\frac{1}{2}.
\end{split}
\end{equation}
Since $supp\omega^\varepsilon(\cdot,0)\subset B_{\varepsilon}(z_0)$ and $\int_{\mathbb{R}^2}{\omega}^\varepsilon(x,0)dx=1$, we have
\begin{equation}\label{1040}
|z_0-m^\varepsilon(0)|=|z_0-\int_{\mathbb{R}^2}x{\omega}^\varepsilon(x,0)dx|\leq \int_{\mathbb{R}^2}|z_0-x|{\omega}^\varepsilon(x,0)dx\leq\varepsilon.
\end{equation}
\eqref{133},\eqref{135} and \eqref{1040} give
\begin{equation}\label{1041}
|\bar{z}^\varepsilon(t)-m^\varepsilon(t)|\leq\varepsilon+L_2\int_0^t|\bar{z}^\varepsilon(s)-m^\varepsilon(s)|ds+2\varepsilon L_2Te^{L_2T},
\end{equation}
then by Gr\"{o}nwall's inequality we have for any $t\in[0,T]$
\begin{equation}
|\bar{z}^\varepsilon(t)-m^\varepsilon(t)|\leq C\varepsilon,
\end{equation}
where $C$ depends only on $L_2$ and $T$. This proves the claim.

Now we finish the proof by showing the following statement: for any $0<\alpha<\frac{1}{3}$, there exists $C>0$, depending only on $\alpha$ and $T$, such that for any $t\in[0,T]$
\begin{equation}\label{1090}
supp\bar{\omega}^\varepsilon\subset B_{C\varepsilon^\alpha}(m^\varepsilon(t)).
\end{equation}
For any fixed $t$, consider a fluid particle $x$, $x\in supp\bar{\omega}^\varepsilon(\cdot,t)$, the growth of the distance between $x$ and $m^\varepsilon(t)$ is
\begin{equation}\label{1050}
\begin{split}
&|(J\nabla\int_{\mathbb{R}^2}\Gamma(x,y)\bar{\omega}^\varepsilon(y,t)dy-\mathbf{F}^\varepsilon(x,t)-\frac{dm^\varepsilon(t)}{dt})\cdot\frac{x-m^\varepsilon(t)}{|x-m^\varepsilon(t)|}|\\
\leq&|(\mathbf{F}^\varepsilon(x,t)+\frac{dm^\varepsilon(t)}{dt})\cdot\frac{x-m^\varepsilon(t)}{|x-m^\varepsilon(t)|}|+|(J\nabla\int_{\mathbb{R}^2}\Gamma(x,y)\bar{\omega}^\varepsilon(y,t)dy)
\cdot\frac{x-m^\varepsilon(t)}{|x-m^\varepsilon(t)|}|\\
\triangleq& J_1+J_2.
\end{split}
\end{equation}
By \eqref{41} and \eqref{130}, $J_1$ can be estimated as follows
\begin{equation}\label{1061}
\begin{split}
J_1&=|(\mathbf{F}^\varepsilon(x,t)+\frac{dm^\varepsilon(t)}{dt})\cdot\frac{x-m^\varepsilon(t)}{|x-m^\varepsilon(t)|}|\\
&\leq|\mathbf{F}^\varepsilon(x,t)+\frac{dm^\varepsilon(t)}{dt}|\\
&=|\mathbf{F}^\varepsilon(x,t)-\int_{\mathbb{R}^2}\bar{\omega}^\varepsilon(y,t)\mathbf{F}^\varepsilon(y,t)dy|\\
&\leq CR,
\end{split}
\end{equation}
where $C$ depends only on $L_2$ and $R\triangleq|x-m^\varepsilon(t)|$.

The term $J_2$ does not contain $\mathbf{F}^\varepsilon$, so the estimate for $J_2$ is exactly the same as the one in Theorem 2.1 in \cite{M}, that is, when $R>C\alpha^\alpha$ for $\alpha<\frac{1}{3}$,
\begin{equation}\label{1062}
J_2\leq C\frac{\varepsilon}{R^2}+A(\varepsilon),
\end{equation}
where $A(\varepsilon)$ is smaller than any power in $\varepsilon$, or equivalently for any $\gamma>0$,
\begin{equation}
\lim_{\varepsilon\rightarrow0}\frac{A(\varepsilon)}{\varepsilon^\gamma}=0.
\end{equation}
\eqref{1061} and \eqref{1062} together give
\begin{equation}
|\frac{dR}{dt}|\leq CR+C\frac{\varepsilon}{R^2}+A(\varepsilon).
\end{equation}
 That is, for any fluid particle at $x$ at time $t$ with $|x-m^\varepsilon(t)|>C\varepsilon^\alpha$, the growth of the distance between $x$ and $m^\varepsilon(t)$ vanishes uniformly in any finite time interval as $\varepsilon\rightarrow0$. Then
\eqref{1090} follows immediately from Gr\"{o}nwall's inequality, which completes the proof.

\end{proof}

Now we are ready to prove Theorem \ref{33}.
\begin{proof}[Proof of Theorem \ref{33}]

Step 1: By definition of $\mathbf{F}^\varepsilon$, we have $\mathbf{F}^\varepsilon(x,t)\equiv0$ for $x\in\{x\in D| dist(x,\partial D)<\frac{\rho_0}{3}\}$. Then by\eqref{46} we know that $\bar{z}^\varepsilon(x,t)\in D_{\frac{\rho_0}{3}}$ for all $t$ and $\varepsilon$, which means that $\bar{z}^\varepsilon$ is uniformly bounded with respect to $\varepsilon$ for $t\in[0,T]$. On the other hand, by Lemma \ref{444}, we know that $\frac{d\bar{z}^\varepsilon(t)}{dt}=-\mathbf{F}^\varepsilon(\bar{z}^\varepsilon(t),t)$ is also uniformly bounded.
 Then by Arzela-Ascoli theorem, we can choose a subsequence of $\{\bar{z}^\varepsilon(t)\}$, say $\{\bar{z}^{\varepsilon_j}(t)\}$, such that $\bar{z}^{\varepsilon_j}(t)\rightarrow \bar{z}(t)$ uniformly for $t\in[0,T]$ as $\varepsilon_j\rightarrow 0$. It is obvious that for all $t\in[0,T]$,
 \begin{equation}\label{161}
 dist(\bar{z}(t),\partial D)>\frac{\rho_0}{3}.
 \end{equation}

Step 2: We show that
\begin{equation}
\mathbf{F}^{\varepsilon_j}(x,t)=J\nabla\int_Dh(x,y)\theta(x)\bar{\omega}^{\varepsilon_j}(y,t)dy
\end{equation}
moreover,
\begin{equation}\label{165}
\mathbf{F}^{\varepsilon_j}(x,t)\rightarrow J\nabla(\theta(x)h(x,\bar{z}(t)))
\end{equation}
uniformly for $t\in[0,T]$.
In fact, by \eqref{245} and \eqref{161} we have
\begin{equation}
dist(supp\bar{\omega}^{\varepsilon_j}(\cdot,t),\partial D)>\frac{\rho_0}{4}
\end{equation}
for any $t\in[0,T]$ provided $\varepsilon_j$ is sufficiently small. Now since $\chi(y)\equiv1$ for any $y\in D_{\frac{\rho_0}{10}}$, we have
\begin{equation}
\begin{split}
\mathbf{F}^{\varepsilon_j}(x,t)&=J\nabla\int_Dh(x,y)\theta(x)\chi(y)\bar{\omega}^{\varepsilon_j}(y,t)dy\\
&=J\nabla\int_{D_\frac{\rho_0}{4}}h(x,y)\theta(x)\chi(y)\bar{\omega}^{\varepsilon_j}(y,t)dy\\
&=J\nabla\int_{D_\frac{\rho_0}{4}}h(x,y)\theta(x)\bar{\omega}^{\varepsilon_j}(y,t)dy\\
&=J\nabla\int_Dh(x,y)\theta(x)\bar{\omega}^{\varepsilon_j}(y,t)dy.
\end{split}
\end{equation}
\eqref{165} can be proved by calculating directly. Indeed,
\begin{equation}
\begin{split}
&|\mathbf{F}^{\varepsilon_j}(x,t)- J\nabla(\theta(x)h(x,\bar{z}(t)))|\\
=&|J\nabla\int_Dh(x,y)\theta(x)\bar{\omega}^{\varepsilon_j}(y,t)dy-J\nabla(\theta(x)h(x,\bar{z}(t)))|\\
=&|\nabla\int_Dh(x,y)\theta(x)\bar{\omega}^{\varepsilon_j}(y,t)dy-\nabla(\theta(x)h(x,\bar{z}(t)))|\\
=&|\int_D\nabla_x(h(x,y)\theta(x))\bar{\omega}^{\varepsilon_j}(y,t)dy-\int_D\nabla_x(\theta(x)h(x,\bar{z}(t)))\bar{\omega}^{\varepsilon_j}(y,t)dy|\\
\leq&\int_D|\nabla_x(h(x,y)\theta(x))-\nabla_x(\theta(x)h(x,\bar{z}(t)))|\bar{\omega}^{\varepsilon_j}(y,t)dy\\
=&\int_{B_\delta(\bar{z}(t))}|\nabla_x(h(x,y)\theta(x))-\nabla_x(\theta(x)h(x,\bar{z}(t)))|\bar{\omega}^{\varepsilon_j}(y,t)dy\\
\leq&\sup_{y\in B_\delta(\bar{z}(t))}|\nabla_x(h(x,y)\theta(x))-\nabla_x(\theta(x)h(x,\bar{z}(t)))|,
\end{split}
\end{equation}
 where $\delta=C\varepsilon^\alpha$. By \eqref{161}, the set $\{y\in B_\delta(\bar{z}(t))|t\in[0,T]\}\subset\subset D$ if $\varepsilon$ is sufficiently small, then $\nabla_x(h(x,y)\theta(x))$ is uniformly continuous on $\mathbb{R}^2\times \{y\in B_\delta(\bar{z}(t))|t\in[0,T]\}$, so
\begin{equation}
\sup_{y\in B_\delta(\bar{z}(t))}|\nabla_x(h(x,y)\theta(x))-\nabla_x(\theta(x)h(x,\bar{z}(t)))|\rightarrow0
\end{equation}
uniformly for $t\in[0,T]$ as $\varepsilon\rightarrow0$, which proves \eqref{165}.

Step 3: Taking limit in \eqref{46}, we know that $\bar{z}(t)$ satisfies the following equation
\begin{equation}
\begin{cases}
 \frac{d\bar{z}(t)}{dt}=-J\nabla(\theta(\bar{z}(t))h(\bar{z}(t),\bar{z}(t))),
 \\ \bar{z}(0)=z_0.
 \end{cases}
\end{equation}

Now recall $z(t)$, the function defined by \eqref{30}. Since $dist(z(t),\partial D)>\rho_0$ and $\theta \equiv1 $ in $D_\frac{\rho_0}{2}$, we can rewrite \eqref{30} as
\begin{equation}\label{181}
\begin{cases}
 \frac{dz(t)}{dt}=-J\nabla(\theta(z(t))h(z(t),z(t))),
 \\ z(0)=z_0.
\end{cases}
\end{equation}
That is, $\bar{z}(t)$ and $z(t)$ satisfy the same ordinary differential equation, then by uniqueness we have $\bar{z}(t)=z(t)$.
 Moreover, we know that the limit is independent of the subsequence $\{\varepsilon_j\}$ and thus $\bar{z}^\varepsilon(t)\rightarrow z(t)$ as $\varepsilon\rightarrow 0$ arbitrarily.

Step 4: We show that $\omega^\varepsilon=\bar{\omega}^\varepsilon$ if $\varepsilon$ is sufficiently small, which will finish the proof of Theorem \ref{33}.

From now on, we choose $x\in supp\omega^\varepsilon(\cdot,0)$ to be fixed. Since $dist(z(t),\partial D)>\rho_0$ for $t\in[0,T]$, \eqref{245} implies $dist(supp{\bar{\omega}^\varepsilon},\partial D)>\frac{2}{3}\rho_0$ by choosing $\varepsilon$ small. In this case $\bar{\Phi}^\varepsilon_t(x)\in D_{\frac{2}{3}\rho_0}$, since $\bar{\Phi}^\varepsilon_0(x)=x\in supp\omega^\varepsilon(\cdot,0)$. So $\theta(\bar{\phi}^\varepsilon_t(x))=1$ for all $t$, which gives the following equations:
\begin{equation}\label{169}
\bar{\omega}^\varepsilon(x,t)=\omega^\varepsilon(\bar{\Phi}_{-t}^\varepsilon(x),0),
 \end{equation}
\begin{equation}\label{175}
 \frac{d}{dt}\bar{\Phi}^\varepsilon_t(x)=\bar{\mathbf{v}}^\varepsilon(\bar{\Phi}^\varepsilon_t(x),t),\,\,\,\,\,\ \bar{\Phi}^\varepsilon_0(x)=x,
 \end{equation}

 \begin{equation}\label{176}
\bar{\mathbf{v}}^\varepsilon(x,t)=J\nabla (\int_{\mathbb{R}^2}\Gamma(x,y)\bar{\omega}^\varepsilon(y,t)dy-\int_{\mathbb{R}^2}h(x,y)\bar{\omega}^\varepsilon(y,t)dy).
\end{equation}

If we define $\hat{\Phi}_t$ as follows
\begin{equation}
\frac{d}{dt}\hat{\Phi}_t(x)=J\nabla G\bar{\omega}^\varepsilon(\hat{\Phi}_t(x),t),\,\,\,\, \hat{\Phi}_0(x)=x,\,\,\,\,\forall x\in D,
\end{equation}
then we find $\bar{\omega}^\varepsilon(x,t)$ and $\hat{\Phi}^\varepsilon_{t}(x)$ satisfy

 \begin{equation}\label{202}
 \bar{\omega}^\varepsilon(x,t)=\omega^\varepsilon(\hat{\Phi}^\varepsilon_{-t}(x),0),
\end{equation}

 \begin{equation}\label{203}
 \frac{d}{dt}\hat{\Phi}_t(x)=J\nabla G\bar{\omega}^\varepsilon(\hat{\Phi}_t(x),t),\,\,\,\, \hat{\Phi}_0(x)=x,\,\,\,\,\forall x\in D,
 \end{equation}
that is, the mapping $t\rightarrow(\bar{\omega}^\varepsilon(\cdot,t),J\nabla G\bar{\omega}^\varepsilon(\cdot,t),\hat{\Phi}^\varepsilon_t(\cdot))$ is also a weak solution of the Euler equation with initial vorticity $\omega^\varepsilon(x,0)$, then by the uniqueness of the Euler equation we have  $\omega^\varepsilon=\bar{\omega}^\varepsilon$, which concludes the proof.

\end{proof}

\begin{remark}
From the above proof we know that for any fluid particle $x$ with initial position  $x\in supp\omega^\varepsilon(x,0)$, the two kinds of motions governed by the Euler equation and the regularized system are the same, but the velocities of these two systems may be different, especially for the fluid particles near the boundary.
\end{remark}

\section{$k$ vortices}

Using the same idea, we can study the evolution of $k(k\geq 1)$ blobs of concentrated vorticity in bounded domains.

According to the vortex model, the evolution of $k$ point vortices with vorticity strength $a_i$ and initial position $z_{i0}$ is described by the following Kirchhoff-Routh equations:
\begin{equation}\label{200}
\begin{cases}
\frac{dz_i(t)}{dt}=\sum_{j\neq i}a_jJ\nabla_{z_i}G(z_i(t),z_j(t))-a_iJ\nabla_{z_i} h(z_i(t),z_j(t)),\\
z_i(0)=z_{i0},
\end{cases}
\end{equation}
where $i=1,2,...,k$, $z_{i0}\in D$, and $z_{i0}\neq z_{j0}$ if $i\neq j$. It means that the motion of the vortex located at $z_i(t)$ with vorticity strength $a_i$ is influenced by the other vortices via the term $\sum_{j\neq i}a_jJ\nabla_{z_i}G(z_i(t),z_j(t))$, and the boundary via the term $-a_iJ\nabla_{z_i} h(z_i(t),z_j(t))$.

By the theory of ordinary differential equations, there exists $T>0$ such that \eqref{200} has a unique solution $\{z_i(t)\}_{i=1}^k$ in time interval $[0,T]$ and $z_{i}(t)\neq z_{j}(t)$ for all $t\in[0,T]$ if $i\neq j$. Moreover, since the set $\{z_i(t)|t\in[0,T]\}$ is compact for each $i$, we choose $\rho_0>0$ such that
\begin{equation}
\rho_0<\min_{t\in[0,T],i\neq j}\{|z_i(t)-z_j(t)|,\,\,\, \rho_0<\min_{t\in[0,T],1\leq i\leq k}\{dist(z_i(t),\partial D)\}.
\end{equation}

Now we consider a family of initial data $\omega^\varepsilon(x,0)\in L^\infty(D)$ satisfying
\begin{equation}\label{201}
\omega^\varepsilon(x,0)=\sum_{i=1}^k\omega_i^\varepsilon(x,0),
\end{equation}
\begin{equation}\label{202}
\int_D\omega_i^\varepsilon(x,0)dx=a_i,
\end{equation}
\begin{equation}\label{203}
|\omega_i^\varepsilon(x,0)|\leq M\varepsilon^{-\eta},\,\,\,\,\eta<\frac{8}{3},
\end{equation}
\begin{equation}\label{204}
supp\omega_i^\varepsilon(x,0)\subset B_\varepsilon(z_{i0}),
\end{equation}
where $M$ and $\eta$ are fixed positive numbers.

For fixed $\varepsilon$, there exists a unique weak solution $(\omega^\varepsilon(x,t),\mathbf{v}^\varepsilon(x,t),\Phi^\varepsilon(x,t))$. Obviously,
\begin{equation}
\omega^\varepsilon(x,t)=\omega^\varepsilon(\Phi^\varepsilon(x,-t),0)=\sum_{i=1}^k\omega^\varepsilon_i(\Phi^\varepsilon(x,-t),0).
\end{equation}
For simplicity we write $\omega^\varepsilon_i(x,t)=\omega^\varepsilon_i(\Phi^\varepsilon(x,-t),0)$, thus $\omega^\varepsilon(x,t)=\sum_{i=1}^k\omega^\varepsilon_i(x,t)$.

We have the following result:
\begin{theorem}\label{210}
For any $\delta>0$, there exists $\varepsilon_0>0$ such that, if $\varepsilon<\varepsilon_0$, then
\begin{equation}\label{29}
supp\omega_i^\varepsilon(x,t)\subset B_\delta(z_i(t))
\end{equation}
for all $0\leq t\leq T$, where $z_i(t)$ is the solution of the following system:
\begin{equation}
\begin{cases}
\frac{dz_i(t)}{dt}=\sum_{j\neq i}a_jJ\nabla_{z_i}G(z_i(t),z_j(t))-a_iJ\nabla_{z_i} h(z_i(t),z_j(t)),\\
z_i(0)=z_{i0}.
\end{cases}
\end{equation}
\end{theorem}
\begin{proof}
Since the proof is essential the same as Theorem \ref{33}, we only show the main idea.
Consider the following regularized system in all of $\mathbb{R}^2$:
\begin{equation}\label{211}
\bar{\omega}^\varepsilon(x,t)=\sum_{i=1}^{k}\bar{\omega}_i^\varepsilon(x,t),
\end{equation}

\begin{equation}\label{212}
\bar{\omega}_i^\varepsilon(x,t)=\omega_i(\bar{\Phi}_i(x,-t),0)
\end{equation}

\begin{equation}\label{213}
\frac{d}{dt}\bar{\Phi}_i^\varepsilon(x,t)=\bar{\mathbf{v}}^\varepsilon_i(\bar{\Phi}^\varepsilon_i(x,t),t),\,\,\, \bar{\Phi}^\varepsilon_i(x,0)=x,\,\forall x\in \mathbb{R}^2,
\end{equation}
\begin{equation}\label{214}
\begin{split}
\bar{\mathbf{v}}^\varepsilon_i(x,t)=J\nabla_x\int_{\mathbb{R}^2}-\frac{1}{2\pi}\ln|x-y|\bar{\omega}^\varepsilon_i(y,t)dy\\
+J\nabla_x\sum_{j\neq i}\int_{\mathbb{R}^2}-\frac{1}{2\pi}\ln^{\rho_0}(|x-y|)\bar{\omega}^\varepsilon_j(y,t)dy
-J\nabla_x\sum_{j=1}^k\int_{\mathbb{R}^2}\bar{h}(x,y)\bar{\omega}^\varepsilon_j(y,t)dt,
\end{split}
\end{equation}
where $\ln^{\rho_0}(|x|)$ is a smooth function of $x$  which coincides with $\ln(|x|)$ if $|x|\geq\frac{\rho_0}{100}$, and $\bar{h}(x,y)=h(x,y)\theta(x)\chi(y)$, where $\theta,\chi$ is defined by \eqref{61},\eqref{62}.

We view the term
\begin{equation}
J\nabla_x\sum_{j\neq i}\int_{\mathbb{R}^2}-\frac{1}{2\pi}\ln^{\rho_0}(|x-y|)\bar{\omega}^\varepsilon_j(y,t)dy
-J\nabla_x\sum_{j=1}^k\int_{\mathbb{R}^2}\bar{h}(x,y)\bar{\omega}^\varepsilon_j(y,t)dt
\end{equation}
as an external force, which is uniformly bounded and satisfies uniform Lipschitz condition. Then by the same argument we can show that $\bar{\omega}_i^\varepsilon$ has localization property, i.e., the support of $\bar{\omega}^\varepsilon_i(\cdot,t)$ is located in a sufficiently small disk centered at $z_i(t)$ if $\varepsilon$ is sufficiently small, but in this case the time evolution of $\bar{\omega}^\varepsilon_i(\cdot,t)$ coincides with the one of ${\omega}^\varepsilon_i(\cdot,t)$, so by uniqueness we have $\bar{\omega}^\varepsilon_i={\omega}^\varepsilon_i$, which concludes the proof.
\end{proof}

\end{document}